\documentclass[12pt]{amsart}
\usepackage {amscd}
\usepackage[german,english]{babel}
\usepackage {amsfonts}
\usepackage[all]{xy} 
\usepackage{amssymb,amsmath, amsthm,latexsym,verbatim}
\usepackage{stmaryrd}
\usepackage{a4wide}
\usepackage[hypertex]{hyperref}\usepackage{enumitem}

\font\k=cmr7
\newcommand {\free}{\mbox{\k free}}
\newcommand {\tors}{\mbox{\k tors}}

\newcommand{\C}{\mathbb{C}}
\newcommand{\R}{\mathbb{R}}

\newcommand{\Z}{\mathbb{Z}}
\newcommand{\Q}{\mathbb{Q}}
\newcommand{\bH}{\mathbb{H}}

\newcommand{\N}{\mathbb{N}}

\newcommand {\rom}{\rho(m)}
\newcommand{\bs}{\backslash}

\newcommand{\rnk}{\operatorname{rk}}
\newcommand{\Mat}{\operatorname{Mat}}
\newcommand{\kL}{\mathfrak{k}}
\newcommand{\tL}{\mathfrak{t}}

\newcommand{\gL}{\mathfrak{g}}

\newcommand{\hL}{\mathfrak{h}}

\newcommand{\cM}{\mathcal{M}}
\newcommand{\cO}{\mathcal{O}}
\newcommand{\Gal}{\operatorname{Gal}}
\newcommand{\Rep}{\operatorname{Rep}}
\newcommand{\Gl}{\operatorname{GL}}
\newcommand{\Aut}{\operatorname{Aut}}
\newcommand{\Id}{\operatorname{Id}}
\newcommand{\sign}{\operatorname{sign}}

\newcommand{\Hom}{\operatorname{Hom}}

\newcommand{\Res}{\operatorname{Res}}

\newcommand{\SL}{\operatorname{SL}}
\newcommand{\GL}{\operatorname{GL}}
\newcommand{\SO}{\operatorname{SO}}
\newcommand{\PSO}{\operatorname{PSO}}
\newcommand{\OO}{\operatorname{O}}
\newcommand{\PO}{\operatorname{PO}}

\newcommand{\vol}{\operatorname{vol}}

\newcommand{\Int}{\operatorname{Int}}
\newcommand{\diag}{\operatorname{diag}}

\newcommand{\rank}{\operatorname{rank}}

\newtheorem {thrm}{Theorem}[section]
\newtheorem {prop}[thrm] {Proposition}
\newtheorem {lem}[thrm] {Lemma}

\theoremstyle{definition}

\theoremstyle{remark}
\newtheorem {bmrk}[thrm] {Remark}

\newtheorem {conj}[thrm]{Conjecture}

\begin{document}
\title[]{On the growth of torsion  in the cohomology of
arithmetic groups}

\date{\today}

\author{Werner M\"uller}
\address{Universit\"at Bonn\\
Mathematisches Institut\\
Endenicher Allee 60\\
D -- 53115 Bonn, Germany}
\email{mueller@math.uni-bonn.de}

\author{Jonathan Pfaff}
\address{Universit\"at Bonn\\
Mathematisches Institut\\q
Endenicher Alle 60\\
D -- 53115 Bonn, Germany}
\email{pfaff@math.uni-bonn.de}

\begin{abstract}
In this paper we consider certain families of arithmetic subgroups of
$\SO^0(p,q)$ and $\SL_3(\R)$, respectively. We study the cohomology of
such arithmetic groups with coefficients in arithmetically defined modules.
We show that for natural sequences of such modules the torsion in the cohomology
grows exponentially.

\end{abstract}

\maketitle

\section{Introduction}

Let $G$ be a semi-simple connected algebraic group over $\Q$, $K$ a 
maximal compact  subgroup of its group of real points $G(\R)$. Let 
$\widetilde X=G(\R)/K$ be the associated Riemannian symmetric space.  Let
$\gL$ and $\kL$ be the Lie algebras of $G(\R)$ and $K$, respectively. Put
$\delta(\widetilde X)=\rank(\gL_\C)- \rank(\kL_\C)$.
Sometimes $\delta(\widetilde X)$ is called the fundamental rank. 
Let $\Gamma\subset G(\Q)$
be an arithmetic subgroup and $X=\Gamma\bs \widetilde X$ the corresponding 
locally
symmetric space. We assume that $G$ is anisotropic over $\Q$, which implies 
that $\Gamma$ is cocompact in $G(\R)$. Let $M$ be an
arithmetic $\Gamma$-module, which means that $M$ is a finite rank free 
$\Z$-module, and there exists an algebraic representation of $G$ on 
$M\otimes_\Z \Q$ such that $\Gamma$ preserves $M$. Then the cohomology
$H^*(\Gamma,M)$ is a finite rank $\Z$-module. Note that if $\Gamma$ is torsion
free, then $H^*(\Gamma,M)\cong H^*(\Gamma\bs \widetilde X,\cM)$, where $\cM$ 
is the local system of free $\Z$-modules associated to $M$.  

For arithmetic reasons, one expects that if $\delta(\widetilde X)=0$, there is 
little
torsion in $H^*(\Gamma,M)$ and the free part dominates the cohomology. On the
other hand, if $\delta(\widetilde X)=1$, one 
expects a lot of torsion in the cohomology and the free part  
to be small. This has been substantiated  by Bergeron and Venkatesh \cite{BV},
who studied the growth of the torsion if $\Gamma$ varies through a sequence of 
congrunce subgroups $\Gamma_n$ for which the injectivity radius of 
$\Gamma_n\bs \tilde X$ goes to infinity. They showed that if 
$\delta(\tilde X)=1$ and $M$ is strongly acylic, the 
torsion grows exponentially proportional to the volume of 
$\Gamma_n\bs \tilde X$.
Furthermore, for compact oriented hyperbolic 3-manifolds, in \cite{MaM} the 
growth of the torsion has been studied if $\Gamma$ is fixed but the 
$\Gamma$-module $M$ grows. More precisely, let $X=\Gamma\bs\bH^3$ be a compact,
oriented hyperbolic 3-manifold with $\Gamma\subset\SL(2,\C)$. Let $V_m$ be
the holomorphic irreducible representation of $\SL(2,\C)$ of
dimension $m+1$. By \cite{BW} one has $H^*(\Gamma,V_m)=0$.
It was proved in \cite{MaM} that for each even $k\in\N$ there
exists a $\Gamma$-invariant lattice $M_{k}\subset V_{k}$. Then 
$H^p(\Gamma,M_{k})$ is a finite abelian group for all $p$ , and the main result
of \cite{MaM} is the following asymptotic formula 
\begin{equation}\label{asymp1}
\lim_{k\to\infty}\frac{\log |H^2(\Gamma,M_{2k})|}{k^2}=\frac{2}{\pi}\vol(X),
\end{equation}
and the estimation
\begin{equation}\label{asymp2}
\log |H^p(\Gamma,M_{2k})|\ll k\log k,\quad p=1,3.
\end{equation}
Note that $H^0(\Gamma,M_{2k})=0$. 

The goal of the present paper is to study the 
growth of the torsion if $M$ varies, for all compact arithmetic quotients 
$\Gamma\bs \tilde X$ of irreducible symmetric spaces $\tilde X$ with 
$\delta(\tilde X)=1$. By the classification of simple Lie groups, the 
irreducible symmetric spaces with $\delta(\tilde X)=1$ are 
$\widetilde X=\SO^0(p,q)/\SO(p)\times \SO(q)$, for $p,q$ odd, and 
$\widetilde X=\SL(3,\R)/\SO(3)$. 
 
The first family of arithmetic groups that we consider are cocompact arithmetic 
subgroups of $\SO^0(p,q)$ that arise from quadratic forms over totally real
number fields.  More precisely, let $F$ be a totally real finite Galois 
extension of $\Q$ of degree $d>1$. We fix an embedding $F\subset \R$.
Let $Q\colon \R^{p+q}\to\R$ be a non-degenerate quadratic form defined over 
$F$ of signature  $(p,q)$. Assume that all non-trivial Galois conjugates of
$Q$ are positive definite.  Let $G:=\SO_Q\subset\GL_{p+q}$ be the special 
orthogonal group of $Q$, i.e., the subgroup of all elements of determinant one 
leaving $Q$ invariant. This is a connected algebraic group over $F$ and its
group of real points $G(\R)$ is isomorphic to $\SO(p,q)$. 

Let $\mathcal{O}_F$ be the ring of algebraic integers of $F$, and let 
$G_{\cO_F}$ be the group of $\cO_F$-valued points of $G$. Then $G_{\cO_F}$ is a 
discrete
cocompact subgroup of $G(\R)$. Via the isomorphism $G(\R)\cong 
\SO(p,q)$, it
corresponds to a discrete cocompact subgroup $\Gamma_0$ of $\SO(p,q)$ (see 
section \ref{SO(p,q)}). If we pass to an appropriate  subgroup of finite 
index $\Gamma\subset \Gamma_0$, we may assume that $\Gamma$ is torsion free 
and that it is containd in $\SO^0(p,q)$. 

Since $G$ is only defined over $F$, we
need to generalize the notion of an arithmetic $\Gamma$-module. Let 
$G^\prime=\Res_{F/\Q}(G)$ be the algebraic $\Q$-group obtained from $G$ by
restriction of scalars. Let $\Delta\colon G\to G^\prime$ be the diagonal 
embedding. Consider $\Gamma$ as a subgroup of $G_{\cO_F}$. Let 
$\Gamma^\prime=\Delta(\Gamma)$. Then $\Gamma^\prime\subset
G^\prime(\Q)$ is an arithmetic subgroup. Let $M$ be an arithmetic 
$\Gamma^\prime$-module. Since $\Gamma\cong\Gamma^\prime$, it becomes a 
$\Gamma$-module and we also call it an arithmetic $\Gamma$-module.   

To state our main result, we need to introduce some notation. 
Let $\widetilde X^d$ be the compact dual symmetric space of $\widetilde X$. 
We chose  an $\SO^0(p,q)$-invariant metric on $\widetilde X$ and equip $X$
and $\widetilde X^d$ with the induced metrics. 
Assume that $p,q$ are odd, $p\geq q$, $p>1$ and let $n:=(p+q)/2-1$. We let
$\epsilon(q):=0$ for
$q=1$ and $\epsilon(q):=1$ for $q>1$ and we put 
\begin{equation}\label{Cpq}
C_{p,q}=\frac{(-1)^{\frac{pq-1}{2}}2^{\epsilon(q)}\pi}{\vol(\widetilde{X}^d)}
\begin{pmatrix}
n\\
\frac{p-1}{2} \end{pmatrix}.
\end{equation}
Then our first main result is the following theorem.

\begin{thrm}\label{Thrm1}
Let $F$ be a totally real Galois extension of $\Q$ of degree $d>1$.
Let $\Gamma$ be a torsion free cocompact arithmetic subgroup of $\SO^0(p,q)$
derived from a quadratic form $Q$ over $F$ as above. Then there exists a 
sequence of arithmetic $\Gamma$-modules $M_m$, $m\in\N$, with the following
properties. The rank $\rnk_{\Z}(M_m)$ is a polynomial in $m$ and there exists
$C>0$, which depends only on $n$, such that
\begin{align*}
\rnk_{\Z}(M_m)=C\: d\: m^{dn(n+1)/2}+O(m^{dn(n+1)/2-1})
\end{align*}
as $m\to\infty$. Furthermore each cohomology group $H^j(\Gamma,M_m)$ is finite 
and 
\begin{equation*}
\sum_{j\geq
0}(-1)^j\log\left|H^j(\Gamma,M_m)\right|=-C_{p,q}
\vol(\Gamma\bs\widetilde X)\,m\rnk_{\Z}(M_m)+O(\rnk_{\Z}(M_m))
\end{equation*}
as $m\to\infty$. 
\end{thrm}

Let $k=(\dim(\widetilde X)+1)/2$. Then it follows from Theorem 
\ref{Thrm1} that there exists a constant $\tilde C_{p,q}>0$, which depends only 
on $p,q$, such that
\begin{equation}
\liminf_m\sum_{j\equiv k (2)}\frac{\log|H^j(\Gamma,M_m)|}{m^{dn(n+1)/2+1}}
\ge \tilde C_{p,q} \,d\vol(\Gamma\bs\widetilde X).
\end{equation}

Thus there is at least one $j$ for which $|H^j(\Gamma,M_m)|$ grows
exponentially in $m$. Given \eqref{asymp1} and \eqref{asymp2}, one is
tempted to pose the following conjecture.
\begin{conj}
Let $\Gamma$ and $M_m$, $m\in\N$, be as above. Then
\[
\lim_{m\to\infty}\frac{\log |H^j(\Gamma,M_m)|}{m^{dn(n+1)/2+1}}=\begin{cases}
\tilde C_{p,q}\,d\vol(\Gamma\bs\widetilde X) ,& 
j=(\dim(\widetilde X)+1)/2,\\0, & 
\mathrm{else}.
\end{cases}
\]
\end{conj}

The next case that we consider are arithmetic subgroups of $\SL_3(\R)$ which 
arise from
9-dimensional division algebras $D$ over $\Q$. Let $\mathfrak{o}$ be an order 
in $D$. Then
$\mathfrak{o}$ induces an arithmetic subgroup $\Gamma$ of $\SL_3(\R)$ which is
cocompact (see section \ref{secsldrei}). After passing to a subgroup of finite 
index, we may assume that $\Gamma$ is torsion-free. 

Let $\tL_\C$ be the standard complexified Cartan-subalgebra of the Lie algebra
of $\SL_3(\R)$ 
equipped with the standard ordering of the roots and let
$\omega_1,\omega_2\in\tL_\C^*$ be the 
corresponding fundamental weights, see \eqref{fweights}. If
$\Lambda=\tau_1\omega_1+\tau_2\omega_2\in\tL_\C^*$,
$\tau_1,\tau_2\in\mathbb{N}$ is a dominant 
weight and $\tau_\Lambda$ is the corresponding irreducible finite-dimensional
representation of $\SL_3(\R)$, we let $\Lambda_\theta$ be the highest weight of
$\tau_\Lambda\circ\theta$. One has
$\Lambda_\theta=\tau_2\omega_1+\tau_1\omega_2$. Moreover, for $m\in\mathbb{N}$
we let 
$\tau_\Lambda(m)$ be the irreducible representation of 
$\SL(3,\R)$ on a complex vector space $V_\Lambda(m)$ with highest weight 
$m\Lambda$. We regard $V_\Lambda(m)$ as 
a real vector-space. 
Let $\widetilde X=\SL(3,\R)/\SO(3)$, let $\widetilde{X}_d$ be the 
compact dual of $\widetilde{X}$ and let $X=\Gamma\bs \widetilde X$.  We fix a
$G$-invariant
metric on $\widetilde{X}$ which induces metrics on $X$ and on $\widetilde{X}_d$.
Then our
main result 
for the $\SL_3(\R)$-case is the following 
theorem. 
\begin{thrm}\label{Thrm2}
Let $\Gamma$ be an arithmetic subgroup of $\SL_3(\R)$ which arises from a
9-dimensional division algebra over $\Q$.  
Let  $\Lambda\in\tL^*_\C$ be a highest weight with $\Lambda_\theta\neq\Lambda$.
Then for each $m$ there exists a lattice $M_\Lambda(m)$ in
$V_\Lambda(m)$ which is stable under $\Gamma$. Moreover, each cohomology
group $H^p(\Gamma,M_\Lambda(m))$ is 
finite and one has
\begin{equation}
\begin{split}
\sum_{p=0}^5(-1)^p&\log\left|H^p(\Gamma,M_\Lambda(m))\right|\\
&=-\frac{\pi\vol(X)}{\vol(\widetilde X_d)}
C(\Lambda)\,m\cdot\rnk_{\Z}M_\Lambda(m)+O(\rnk_{\Z}M_\Lambda(m)),
\end{split}
\end{equation}
as $m\to\infty$, where $C(\Lambda)>0$ is a constant which depends only on
$\Lambda$. 
If $\Lambda$ equals one of the fundamental weights 
$\omega_{f}$ then $C(\Lambda)=4/9$.
\end{thrm}
The rank of $M_\Lambda(m)$ can be computed explicitly as follows. Firstly,
if 
$\Lambda$ is equal to one of the fundamental weights $\omega_1$ or $\omega_2$,
then 
$\Lambda_\theta\neq\Lambda$ and Weyl's dimension formula gives
\begin{align*}
\rnk_{\Z}M_\Lambda(m)=\dim_{\R}(V_\Lambda(m))=\frac{m^2}{2}+O(m),
\end{align*}
as $m\to\infty$. Secondly, if $\Lambda=\tau_1\omega_1+\tau_2\omega_2$,  
$\tau_1,\tau_2\in\mathbb{N}$, $\tau_1\tau_2\neq 0$, then the condition
$\Lambda\neq\Lambda_\theta$ 
is equivalent to $\tau_1\neq \tau_2$ and again by Weyl's dimension formula one
has
\begin{align*}
\rnk_{\Z}M_\Lambda(m)=\dim_{\R}(V_\Lambda(m))=\frac{\tau_1^2\tau_1
+\tau_2^2\tau_1}{2}m^3+O(m^2),
\end{align*}
as $m\to\infty$. Let $M_{i,m}:=M_{\omega_i}(m)$, $i=1,2$. Then it 
follows that
\begin{equation}
\liminf_m\sum_{j=0}^2\frac{\log|H^{2j+1}(\Gamma,M_{i,m})|}{m^3}\ge 
\frac{2\pi}{9\vol(\widetilde X_d)}\vol(X).
\end{equation}
Again, by \eqref{asymp1} and \eqref{asymp2}, one is led to the
following conjecture.
\begin{conj} Let $\Gamma$ and $M_{i,m}$, $m\in\N$, be as above. Then one has
\begin{equation}
\lim_{m\to\infty}\frac{\log|H^3(\Gamma,M_{i,m})|}{m^3}=
\frac{2\pi}{9\vol(\widetilde X_d)}\vol(X)
\end{equation}
and
\begin{equation}
\log|H^j(\Gamma,M_{i,m})|=o(m^3),\quad j\neq 3.
\end{equation}
\end{conj}
There are similar statements for highest  weights $\Lambda=\tau_1\omega_1+
\tau_2\omega_2$ with $\tau_1\tau_2\neq0$.

Next we describe our approach to prove the main results. As in \cite{MaM}, 
it is based on the study of the analytic torsion. To begin with we consider an
arbitrary connected semi-simple algebraic group $G$ over $\Q$. Let the 
notation be as at the beginning of the introduction. Assume that 
$\delta(\widetilde X)=1$. Choose a $G$-invariant 
Riemannian metric $g$ on $\widetilde X$. Assume that $\Gamma\subset G(\Q)$ is
torsion free. Let $X:=\Gamma\bs\widetilde X$ equipped with the metric induced 
by $G$. Then we have $H^*(\Gamma,M)= H^*(X,\cM)$. Let $V:=M\otimes_\Z\R$ and 
let $\rho\colon G(\R)\to \GL(V)$ be the 
representation associated to the arithmetic $\Gamma$-module $M$. 
Let $E\to X$ be the flat vector bundle associated to  $\rho|_\Gamma$. Choose a 
Hermitian fibre metric in $E$. Let $T_X(\rho)\in\R^+$ be the analytic torsion 
of $X$ and $E$. Recall that
\begin{equation}\label{anator}
\log T_X(\rho)=\frac{1}{2}\sum_{p=1}^n(-1)^pp\frac{d}{ds}
\zeta_p(s;\rho)\big|_{s=0},
\end{equation}
where $\zeta_p(s;\rho)$ is the zeta function of the Laplace
operator  $\Delta_p(\rho)$ on $E$-valued $p$-forms and $n=\dim X$ 
(see \cite{Mu1}). Assume that $\rho$ is acyclic, that is $H^*(X,E)=0$. 
Then $T_X(\rho)$ is metric independent \cite[Corollary 2.7]{Mu1} and  equals 
the Reidemeister torsion $\tau_X(\rho)$ \cite[Theorem 1]{Mu1}. Moreover,
$H^*(\Gamma,M)$ is a finite abelian group. Using Proposition \ref{Proptors}, 
we get
\begin{equation}\label{cohotor}
\log T_X(\rho)=\sum_{q=0}^n (-1)^{q+1}\log |H^q(\Gamma,M)|.
\end{equation}
This is the key equality which we apply to prove Theorems \ref{Thrm1} and
\ref{Thrm2}. In \cite{MP} we studied the asymptotic behavior of $T_X(\tau_m)$
for certain sequences of irreducible representations of $G(\R)^0$. We will apply
the results of \cite{MP} to our case. The main issue is to construct 
appropriate arithmetic $\Gamma$-modules. 

We start with the case of $\widetilde X=\SO^0(p,q)/\SO(p)\times \SO(q)$, $p,q$ 
odd. Let $G=\SO_Q$ be the special orthogonal group of a quadratic form $Q$
over a totlally real number field $F$ as defined above. Then $G$ is a 
connected algebraic group over $F$. Let $G^\prime=\Res_{F/\Q}(G)$ be the
algebraic $\Q$-group obtained from $G$ by restriction of scalars \cite{Weil}. 
The group of real points $G^\prime(\R)$ is given by
\[
G^\prime(\R)\cong \SO(p,q)\times K_1,
\]
where $K_1$ is the product of $d-1$ copies of $\SO(p+q)$. Then we construct
a sequence $\rho(m)\colon G^\prime\to GL(V_m)$, $m\in\N$, of $\Q$-rational 
representations such that the irreducible components of $\rho(m)(\R)\colon
G^\prime(\R)^0\to \GL(V_m\otimes_\Q\R)$ are of the form considered in 
\cite[Theorem 1.1]{MP}. Let $\Delta\colon G\to G^\prime$ be the diagonal 
embedding. Let $\Gamma^\prime=\Delta(\Gamma)$. Then $\Gamma^\prime$ is an 
arithmetic subgroup of $G^\prime(\Q)$. Therefore, $V_m$
contains a  lattice $M_m$ which is invariant under $\rho(m)(\Gamma^\prime)$. 
Through the isomorphism $\Gamma\cong\Gamma^\prime$, $M_m$ becomes a 
$\Gamma$-module. This is our arithmetic $\Gamma$-module. By construction we
have 
$H^*(\Gamma,M_m)\cong H^*(\Gamma^\prime,M_m)$. Thus it suffices to prove the
statement of Theorem \ref{Thrm1} for $\Gamma^\prime$. 

Let $K^\prime=\SO(p)\times\SO(q)\times K_1$. Then $K^\prime$ is a maximal
compact 
subgroup of $G^\prime(\R)^0$. Let $\widetilde X^\prime=G^\prime(\R)^0/K^\prime$
and
$X^\prime:=\Gamma^\prime\bs\widetilde X^\prime$. Now we apply
\cite[Propositions 1.2, 1.3]{MP} to determine the asymptotic behavior of 
$T_{X^\prime}(\rho(m))$ as $m\to\infty$. 
Finally we use \eqref{cohotor} to establish Theorem \ref{Thrm1}. 

The proof of Theorem \ref{Thrm2} uses similar arguments, which are also based 
on \eqref{cohotor} and \cite{MP}. 

The paper is organized as follows. In section \ref{prelim} we collect some
facts about cohomology of fundamental groups of manifolds with coefficients in 
a free $\Z$-module. We also recall some elementary facts about algebraic groups.
 In section \ref{SO(p,q)} we consider arithmetic subgroups 
of $\SO^0(p,q)$ and prove Theorem \ref{Thrm1}. The prove of Theorem 
\ref{Thrm2} is the content of the final section \ref{secsldrei}.

\section{preliminaries}\label{prelim} 

Let $X$ be a closed connected smooth manifold of dimension $d$. Let 
$\Gamma:=\pi_1(X,x_0)$ be the 
fundamental group of $X$ with respect to some base point $x_0$ and let 
$\widetilde X$ be the correpsonding universal covering. Thus $\Gamma$ acts
properly discontinuously and fixed point free on $\widetilde X$ and 
$X=\Gamma\bs\widetilde X$. Assume that $\widetilde X$ is contractible.

Let $M$ be a free finite-rank $\Z$-module and let $\rho$ be a representation of
$\Gamma$ on $M$. Let $H^q(\Gamma,M)$ be 
the $q$-th cohomology group of $\Gamma$ with coefficients in $M$, see \cite{Br}.
These groups can be computed as follows. Let $L$ be a smooth 
triangulation of $X$ and let $\widetilde L$ 
be the lift of $L$ to a triangulation of $\widetilde{X}$. Let 
$C_q(\widetilde{L};\Z)$ be the free abelian group generated by the $q$-chains 
of $\widetilde{L}$, let $C^q(\widetilde{L};\Z):=
\Hom_{\Z}(C_q(\widetilde{L},\Z);\Z)$ and 
let $C_*(\widetilde{L};\Z)$ resp. $C^*(\widetilde{L};\Z)$ be the associated 
simplical chain- resp. cochain complexes. Each $C_q(\widetilde{L};\Z)$ is a
free 
$\Z[\Gamma]$ module and if one fixes an embedding of $L$ into 
$\widetilde{L}$, then the $q$-cells of $L$ form a base of 
$C_q(\widetilde{L};\Z)$ over $\Z[\Gamma]$. Let  
\begin{align*}
C^q(L,M):=C^q(\widetilde{L};\Z)\otimes_{\Z[\Gamma]}M. 
\end{align*}
Then the $C^q(L,M)$ form again a cochain complex $C^*(L,M)$ and 
the corresponding cohomology groups will be denoted by $H^q(L,M)$ . There
is an isomorphism 
\begin{align*}
C^q(L,M)\cong\Hom_{\Z[\Gamma]}(C_q(\widetilde{L};\Z),M),
\end{align*}
which induces an isomorphism of the corresponding cochain complexes.
Since $\widetilde{L}$ is contractible, the complex $C_*(\widetilde{L})$ is a 
free resolution of $\Z$ over $\Z[\Gamma]$ and thus one has 
\begin{equation}\label{cohoGamma}
H^q(\Gamma,M)\cong H^q(\Hom_{\Z[\Gamma]}(C_{*}(\widetilde{L}),M))\cong H^q(L,M).
\end{equation}
Each cohomology group $H^q(L,M)$ is a finitely generated abelian group.
Let $H^q(\Gamma,M)_{\tors}$ be the torsion subgroup of $H^q(\Gamma,M)$ and
let $H^q(\Gamma,M)_{\free}=H^q(\Gamma,M)/H^q(\Gamma,M)_{\tors}$ be the free
part.
Then one has
\begin{equation*}
H^q(\Gamma,M)=H^q(\Gamma,M)_{\free}\oplus H^q(\Gamma,M)_{\tors}.
\end{equation*}

Now let $V:=M\otimes_\Z\C$ and $V_\R:=M\otimes_\Z\R$. Then $V$ is a 
finite-dimensional complex vector space, $V_\R\subset V$ is a real structure on
$V$ and $M$ is a lattice in $V_\R$. We regard $\rho$ as a representation of
$\Gamma$ on $V$. Then $\rho$ is unimodular, i.e., $|\det\rho(\gamma)|=1$ for 
all $\gamma\in\Gamma$. Let
\begin{align*}
C^q(L,V):=C^q(\widetilde{L})\otimes_{\Z[\Gamma]}V.
\end{align*} 
The $C^q(L,V)$'s form a chain complex $C^*(L,V)$ of 
finite-dimensional $\C$-vector spaces and one has
\begin{align}\label{tens}
C^q(L,V)=C^q(L,M)\otimes_\Z\C.
\end{align}
Let $E:=\tilde{X}\times_{\rho} V$ be the flat vector bundle over $X$ associated 
to $\rho$. Then by the de Rham isomorphism, the cohomology 
groups $H^q(L,V)$ of the complex $C^*(L,V)$ are canonically 
isomorphic to the cohomology groups $H^q(X,E)$ 
of the complex of $E$-valued differential forms on $X$. By Hodge theory they 
are canonically isomorphic to the space of $E$-valued harmonic forms for any
choice of metrics on $X$ and $E$, respectively. 

We assume that the bundle $E$ is acyclic i.e. that $H^q(L;V)=H^q(X;E)=0$ for
all $q$. This holds in all cases that we study in this paper. Let
$\sigma^q_j$, $j=1,\dots,r_q$, be the oriented $q$-simplices of $L$ considered 
as a preferred basis of the $\Z[\Gamma]$-module $C^q(\widetilde L;\Z)$. Let 
$e_1,\dots,e_m$ be a basis of $M$. Then $\{\sigma_j^q\otimes e_k\colon 
j=1,\dots,r_q,\, k=1,\dots,m\}$ is a preferred basis of $C^q(L;M)$ and also of
$C^q(L;V)$.  Let $\tau_X(\rho)\in\R^+$ be the 
Reidemeister torsion with respect to these volume elements (see \cite{Mu1},
\cite{MaM}). Note that $\tau_X(\rho)=|\tau^\C_X(\rho)|$, 
where $\tau^\C_X(\rho)\in\C^\times$ is the complex Reidemeister torsion. 
Since $\rho$ is acyclic, $\tau_X(\rho)$ is a combinatorial invariant of $X$ 
and $\rho$ which is independent of the choices that we made 
(see \cite[section 1]{Mu1}). Moreover, each cohomology group 
$H^q(\Gamma,M)$ is finite, i.e., $H^q(\Gamma,M)=H^q(\Gamma,M)_{\tors}$ 
and the order $|H^q(\Gamma,M)|$ of these groups is related to the Reidemeister
torsion as follows.

\begin{prop}\label{Proptors}
Assume that $H^q(X,E)=0$ for all $q$. Then one has 
\begin{align*}
\sum_{q=0}^d(-1)^{q+1}\log{|H^q(\Gamma,M)|}=\log\tau_X(\rho).
\end{align*}
\end{prop}
\begin{proof}
Let $C^*(L,V_\R)$ be the chain complex of the finite-dimensional real vector 
spaces
\[
C^q(L,V_\R):=C^q(\widetilde L)\otimes_{\Z[\Gamma]}V_\R,\quad q=0,...,d.
\]
We have
\[
C^q(L,V_\R)=C^q(L,M)\otimes_\Z\R.
\]
Let $\rho_\R\colon \Gamma\to \GL(V_\R)$ be the representation induced by $\rho$
and let $E_\R:=\widetilde X\times_{\rho_\R} V_\R$ be the associated flat real
vector 
bundle. Then $H^*(X;E_\R)=0$. The  basis of the free $\Z$-module $C^q(L;M)$, 
described above, gives rise to a distinguished basis of $C^q(L;V_\R)$. 
Let $\tau_X(\rho_\R)$ be the Reidemeister torsion
of the complex $C^*(L;V_\R)$ with respect to volume elements defined by these
bases. Then it follows from  \eqref{cohoGamma} and \eqref{tens} as in 
\cite[section 2.2]{BV} that
\begin{equation}\label{tortor}
\log\tau_X(\rho_\R)=\sum_{q=0}^d(-1)^{q+1}\log{|H^q(\Gamma,M)|}.
\end{equation}
See also \cite[Proposition 2.3]{MaM} and \cite[Lemma 2.1.1]{Tu}. Since the 
coboundary operators
of the complexes $C^*(L;V_\R)$ and $C^*(L;V)$, respectively, are induced by
the coboundary operators of $C^*(L;M)$, it follows from the definition of 
the Reidemeister torsion that $\tau_X(\rho_\R)=\tau_X(\rho)$. Combined with
\eqref{tortor} the proposition follows.
\end{proof}

Finally we recall some facts conerning linear algebraic groups. For all details
we  refer to \cite{Bo2}. Let $F$ be a finite Galois extension of $\Q$ with
Galois group $\Sigma:=\Gal(F/\Q)$. For $\sigma\in\Sigma$ and $x\in F$
let $x^\sigma$ denote the image of $x$ under $\sigma$. 
If $G$ is a linear algebraic group over $F$ with coordinate 
algebra $F[G]$, let $G^\sigma$ denote the linear algebraic group conjugate 
by $\sigma$, see \cite{Bo2}. If $G$ is realized as the zero set in 
some $F^n$ of an ideal $I$ in $F[X_1,\dots,X_n]$, 
then $G^\sigma$ is the zero set of the ideal $I^\sigma$, where $I^\sigma$ is
obtained from $I$ by applying 
$\sigma$ to each polynomial coefficient. Each $F$-rational homomorphism
$\rho:G\to H$ of linear
algebraic groups over $F$ induces canonically an $F$-rational homomorphism
$\rho^\sigma:G^\sigma\to H^\sigma$. 

If $G$ is an algebraic group defined over $F$, an algebraic group $G'$ defined 
over $\Q$ together with an $F$-rational isomorphism $\mu\colon G'\times_\Q F\to
G$ is called a $\Q$-structure of $G$. The $\Q$-structure canonically
induces an action of $\Sigma$ on the coordinate algebra of $G$ and 
thus on $G$ itself.   

Let $V$ be a finite-dimensional $F$-vector space. A $\Q$-structure $V_0$ of $V$
is a $\Q$-subspace $V_0$ of $V$ such that the embedding $V_0\hookrightarrow V$
induces an isomorphism $V_0\otimes_\Q F$ of $F$-vector spaces. For 
$\sigma\in\Sigma$ a $\Q$-linear automorphism $A$ of $V$ is called 
$\sigma$-linear if $A(\lambda v)=\sigma(\lambda)A(v)$, $\lambda\in F$, $v\in V$.
Then a semi-linear action of $\Sigma$ on $V$ is given by a family 
$\{f_\sigma\}_{\sigma\in\Sigma}$ of $\sigma$-linear automorphisms $f_\sigma$ of
$V$,
satisfying  $f_{\sigma\tau}=f_\sigma\circ f_\tau$, $\sigma,\tau\in\Sigma$.
Given a semi-linear action of $\Sigma$ on $V$, the set 
$V^\Sigma:=\{v\in V\colon f_\sigma(v)=v,\,\forall\sigma\in\Sigma\}$ is 
a $\Q$-structure of $V$ and every $\Q$-structure is of this form
(see \cite[AG.14.2]{Bo2}). If $V_0$ is a
$\Q$-structure of $V$, then $\Gl(V_0)$ is a $\Q$-structure of 
$\Gl(V)$ and the corresponding action of $\Sigma$ on $\Gl(V)$ is given by 
$\sigma\cdot g:=f_\sigma\circ g\circ f_\sigma^{-1}$, $g\in\Gl(V)$.

\section{Arithmetic subgroups of $\SO^0(p,q)$}\label{SO(p,q)}

Let $p,q\in\mathbb{N}$ be odd. Put
\[
p_1=(p-1)/2,\, q_1=(q-1)/2,\, n:=p_1+q_1.
\]
We denote by $\SO(p,q)$ the group of isometries of the standard quadratic 
form of signature $(p,q)$ on $\R^{p+q}$ with determinant $1$. Let $\SO^0(p,q)$
denote the identity component of $\SO(p,q)$. The group
$\SO^0(p,q)$ is of fundamental rank one. Let $\gL$ be the Lie algebra of 
$\SO^0(p,q)$. We choose the fundamental Cartan subalgebra as follows. 
Let $E_{i,j}$ be 
the $(p+q)\times (p+q)$-matrix which is one at the $i$-th row and $j$-th
column and which is zero elsewhere. Let 
\begin{align}
H_1:=E_{p,p+1}+E_{p+1,p}.
\end{align}
and let
\begin{equation}\label{basis1}
H_i:=\begin{cases} \sqrt{-1}(E_{2i-3,2i-2}-E_{2i-2,2i-3}),&2\leq i\leq p_1+1\\
\sqrt{-1}(E_{2i-1,2i}-E_{2i,2i-1})&p_1+1< i\leq n+1.\end{cases}
\end{equation}
Then 
\begin{align*}
\hL:=H_1\oplus\bigoplus_{i=2}^{n+1}\sqrt{-1}H_i
\end{align*}
is a Cartan subalgebra of $\gL$. Define 
$e_{i}\in\mathfrak{h}_{\mathbb{C}}^{*}$, $i=1,\dots,n+1$,  by
\begin{align*}
e_{i}(H_{j})=\delta_{i,j},\: 1\leq i,j\leq n+1.
\end{align*}
The finite-dimensional irreducible complex representations $\tau$ of
$\SO^0(p,q)$ are 
parametrized by their highest weights $\Lambda(\tau)\in\hL_\C^*$ given by
\begin{equation}\label{Darstellungen von G}
\begin{split}
\Lambda(\tau)=&k_{1}(\tau)e_{1}+\dots+k_{n+1}(\tau)e_{n+1},\quad
(k_{1}(\tau),\dots k_{n+1}(\tau))\in{{\mathbb{Z}}
}^{n+1},\\
&k_{1}(\tau)\geq k_{2}(\tau)
\geq\dots\geq k_{n}(\tau)\geq \left|k_{n+1}(\tau)\right|.
\end{split}
\end{equation}
For $\Lambda(\tau)$ a weight as in \eqref{Darstellungen von G}, the highest
weight $\Lambda(\tau_\theta)$ of the representation $\tau\circ\theta$ is
\begin{align}\label{lambdatheta}
\Lambda(\tau_\theta)=k_{1}(\tau)e_{1}+\dots+k_{n}(\tau)e_{n}-k_{n+1}(\tau)e_{
n+1}.
\end{align}
If we let 
\begin{align}\label{hiweight}
\omega^+_{f,n}:=\sum_{j=1}^{n+1}e_j;\quad\omega^-_{f,n}:=(\omega^+_{f
,n})_\theta=\sum_{j=1}^{n}e_j-e_{n+1},
\end{align}
then $\frac{1}{2}\omega^{\pm}_{f,n}$ are the fundamental weights which are 
not invariant under $\theta$. 
We now recall the construction of certain arithmetically defined cocompact 
subgroups of $\SO^0(p,q)$. For more details see 
\cite[section 3.2, Appendix B]{S} and for the $\SO^0(p,1)$-case \cite{Millson}.

Let $F$ be a totally real number field of degree $d=[F:\Q]>1$. Let $\Sigma$ be 
the Galois group of $F$ over $\Q$. We fix an embedding $F\subset \R$. Let
$1\in\Sigma$ be the identity.  
Let $\alpha_j\in F^*$, $j=1,\dots,p+q$, be such that
\[
\sign(\alpha_j)=\begin{cases}+1,& \textup{if}\,\, j\le p,\\
-1,&\textup{if}\,\, p< j\le p+q,\end{cases} 
\]
and
\[
\sign\left(\sigma(\alpha_j)\right)=+1,\quad \sigma\in\Sigma\setminus\{1\},
\,j=1,\dots,p+q.
\]
For $\sigma\in\Sigma$ let $Q^\sigma$ be the quadratic form on $\R^{p+q}$
defined 
by
\[
Q^\sigma(x)=\sum_{j=1}^{p+q}\sigma(\alpha_j) x_j^2.
\]
Then $Q:=Q^{1}$ is a non-degenerate quadratic 
form of signature $(p,q)$ and $Q^\sigma$, $\sigma\neq 1$, is positive definite.

Let $G:=\SO_Q\subset \GL_{p+q}$ be the special orthogonal group of $Q$, 
i.e., the subgroup of all elements of determinant 
one leaving $Q$ invariant. Then $G$ is a connected algebraic 
group defined over $F$. Let $J\in\GL_{p+q}(\R)$ be defined by
\[
J:=\diag\left(\sqrt{\alpha_1},\dots,\sqrt{\alpha_p},\sqrt{-\alpha_{p+1}},\dots,
\sqrt{-\alpha_{p+q}}\right).
\]
Then the map $g\mapsto JgJ^{-1}$ establishes an isomorphism 
$G(\R)\cong\SO(p,q)$.
Similarly, we have $G^{\sigma}(\R)\cong\SO(p+q)$, if $\sigma\neq 1$. Let 
\begin{equation}
G^\prime\cong\Res_{F/\Q}(G)
\end{equation}
be the algebraic $\Q$-group obtained by 
restriction of scalars. There is a canonical isomorphism of algebraic groups 
over $F$
\begin{equation}\label{eqalpha}
\alpha\colon  G^\prime\times_\Q F\cong \prod_{\sigma\in\Sigma} G^{\sigma},
\end{equation}
and the group of real points $G^\prime(\R)$ satisfies
\[
G^\prime(\R)\cong\SO(p,q)\times\prod_{\sigma\in\Sigma\setminus\{1\}}\SO(p+q).
\]
Denote by 
\begin{equation}\label{diag}
\Delta\colon G\to \prod_{\sigma\in\Sigma} G^{\sigma}
\end{equation}
the diagonal embedding given by $\Delta(g)=(g^{\sigma})_{\sigma\in\Sigma}$. 

Let $\mathcal{O}_F$ be the ring of integers of $F$ and let $G_{\mathcal{O}_F}$
be 
the group of $\mathcal{O}_F$-units of $G$. An arithmetic subgroups of 
$G(F)$ 
is by definition a subgroup which is commensurable with $G_{\mathcal{O}_F}$.
Let $\Gamma_0:=JG_{\mathcal{O}_F}J^{-1}$. Then $\Gamma_0$ is a subgroup of 
$\SO(p,q)$.
 
\begin{lem}\label{Lemma1}
$\Gamma_0$ is a discrete, cocompact subgroup of $\SO(p,q)$.
\end{lem}
\begin{proof}
For $\sigma\in\Sigma\setminus\{1\}$, the group $G^{\sigma}(\R)$ is isomorphic 
to $\SO(p+q)$. Thus by \cite[p. 530]{BH}, $\Gamma_0$ is discrete in $\SO(p,q)$.
Since all quadratic forms $Q^{\sigma}$, $\sigma\neq 1$, are positive definite, 
the form $Q$ is anisotropic over $F$. Thus, by \cite[page
256]{Bo2} $G$ is anisotropic over $F$. Therefore, $G_{\mathcal{O}_F}$ contains
no non-trivial unipotent elements. 
Using \cite[Lemma 11.4, Theorem 12.3]{BH}, it follows that the diagonal 
image of $G_{\mathcal{O}_F}$ in $\prod_{\sigma\in\Sigma}G^\sigma(\R)$ 
is cocompact and since $G^{\sigma}(\R)$ is compact for $\sigma\neq 1$,
$\Gamma_0$
is also cocompact in $\SO(p,q)$.
\end{proof}
\begin{bmrk}
If $F=Q[\sqrt{v}]$ is a real quadratic field, then putting
$\alpha_1=\dots=\alpha_p=1$, $\alpha_{p+1}=\dots=\alpha_{p+q}=-\sqrt{v}$ the
above
construction has already been given in \cite[section 4.3]{Bo}.
\end{bmrk}

Now we let $B$ be the symmetric bilinear form on $F^{p+q}$ given by 
\begin{align*}
B(e_{i},e_{j})=\begin{cases}1,&i+j=p+q+1\\ 0,&i+j\neq p+q+1\end{cases}, 
\end{align*}
for $e_1,\dots,e_{p+q}$ the standard base of $F^{p+q}$.
Let $\OO_{B}$ be the orthogonal group of $B$ and let $\SO_{B}$ be 
the elements of $\OO_{B}$ of determinant one. Then $\OO_{B}$ and
$\SO_{B}$ are algebraic groups defined over $F$ and there 
exists an isomorphism $\mu:G(\bar{F})\to\SO_{B}(\bar{F})$, i.e. $G$ is a form
of $\SO_{B}$ over $F$. 

Let $T$ be the maximal torus of $\SO_{B}(\bar{F})$ given by 
\begin{align*}
T=\{\diag(\lambda_1,\dots,\lambda_{n+1},\lambda_1^{-1},\dots,\lambda_{n+1}^{-1})
,\: 
\lambda_1,\dots,\lambda_{n+1}\in \bar{F}^{*}\},
\end{align*} 
where $n=(p+q)/2-1$. 
Then $T$ is defined over $F$. Let $X(T)$ be the character group of $T$,
written additively. Then a base of $X(T)$ is given by the $f_i:T\to \bar{F}$,
$f_i(\diag(\lambda_1,\dots,\lambda_{n+1},\lambda_1^{-1},\dots,\lambda_{n+1}^{-1}
))=\lambda_i$, where $1\leq i\leq n+1$. 

By $\Rep(\SO_{B}(\bar{F}))$ we denote the finite-dimensional 
representations of $\SO_{B}(\bar{F})$ which are irreducible. Then the
elements of $\Rep(\SO_{B}(\bar{F}))$
correspond bijectively to their highest weights 
$\lambda_\tau:=m_1f_1+\dots+m_{n+1}f_{n+1}$, 
where $m_1,\dots,m_{n+1}\in\mathbb{Z}$, $m_1\geq m_2\geq\dots\geq m_{n}\geq
|m_{n+1}|$. Since $T$
is split over $F$, every finite-dimensional irreducible representation 
of $\SO_{B}(\bar{F})$ is defined over $F$, \cite[Proposition 2.3]{Tits}.

For $\tau\in\Rep(\SO_{B}(\bar{F}))$ with highest weight
$\lambda_{\tau}=m_1(\tau)f_1+\dots+m_{n+1}(\tau)f_{n+1}$ 
let $\tau'\in\Rep(\SO_{B}(\bar{F}))$ be the element with highest weight
$\lambda_{\tau'}=m_1(\tau)f_1+\dots+m_{n}f_{n}-m_{n+1}(\tau)f_{n+1}$.
Then the following lemma holds. 

\begin{lem}\label{Lem}
For every $\tau\in\Rep(\SO_{B}(\bar{F}))$ there exists a representation 
$\tilde{\tau}$ of $\OO_{B}(\bar{F})$ that restricts to $\tau+\tau'$ on
$\SO_{B}(\bar{F})$. 
\end{lem}
\begin{proof}
The proof of the corresponding proposition for $\SO_{B}(\C)$ given in 
\cite[Theorem 5.22]{GW} extends withouth difficulty to any algebraically closed
field of characteristic zero.
\end{proof}

\begin{bmrk}
If $\tau$ satisfies $\tau=\tau'$ then there exists in fact a representation 
of $\OO_{B}$ that restricts to $\tau$. However, we are only interested in 
the case $\tau\neq\tau'$ and in this case the representation $\tilde{\tau}$ 
from the previous lemma is irreducible.  
\end{bmrk}

Now we let $\PSO_{B}:=\SO_{B}/\{\pm \Id\}$. Then an element
$\tau\in\Rep(\SO_{B}(\bar{F}))$ 
of highest weight $\lambda_\tau=m_1f_1+\dots+m_{n+1}f_n$ descends to a
representation 
of $\PSO_{B}(\bar{F})$ if and only $m_1+\dots+m_{n+1}$ is even. 

\begin{lem}\label{Lemma2}
Let $\tau$ be a representation of $\SO_{B}$ over $\bar{F}$ which descends
to a representation 
of $\PSO_{B}$ over $\bar{F}$. Then there exists an $F$-rational
representation of $G$ which over $\bar{F}$ is equivalent to
$(\tau+\tau')\circ\mu$. 
\end{lem}
\begin{proof}
For $\sigma\in\Gal(\bar{F}/F)$ define an automorphism $\phi_\sigma$ of
$\SO_{B}(\bar{F})$ 
by $\phi_\sigma:=\mu\circ\sigma\circ\mu^{-1}\circ\sigma^{-1}$. Since the Dynkin
diagram $D_{n+1}$ has exactly one non-trivial automorphism, 
there is a natural isomorphism
$\Aut(\SO_{B}(\bar{F}))\cong\PO_{B}(\bar{F})$, where 
$\PO_{B}(\bar{F})$ acts on $\SO_{B}(\bar{F})$ by conjugation, and
thus 
there exists a unique $a_\sigma\in\PO_{B}(\bar{F})$ such that for each
$g\in \SO_{B}(\bar{F})$ one 
has $\phi_\sigma(g)=a_\sigma g a_\sigma^{-1}$. Thus one has 
\begin{align}\label{eq1}
\mu(\sigma(g))=a_\sigma\sigma(\mu(g))a_\sigma^{-1}.
\end{align}
We can regard the assignment $\sigma\to a_\sigma$ 
as an element of the first Galois-cohomology set
$H^1(\Gal(\bar{F}/F),\PO_{B}(\bar{F}))$.
By Lemma \ref{Lem} there exist a representation $\tilde{\tau}$ of 
$\PO_{B}(\bar{F})$ on $V_{\tilde{\tau}}=V_{\tau}\oplus V_{\tau'}$
which restricts to $\tau\oplus\tau'$ 
on $\SO_{B}(\bar{F})$. The assignment
$\sigma\mapsto\tilde{\tau}(a_\sigma)$ 
is an element of $H^1(\Gal(\bar{F}/F),\Gl(V_{\tilde{\tau}}))$ and since this set
is 
trivial by Hilbert's theorem 90, there exists an $x\in \Gl(V_{\tilde{\tau}})$
such
that 
\begin{align}\label{eq2}
\tilde{\tau}(a_\sigma)=x^{-1}\sigma(x)\quad\forall \sigma\in \Gal(\bar{F}/F).
\end{align}
Now define a representation $\rho$ of $G(\bar{F})$ by $\rho:=\Int(x)\circ
(\tau+\tau')\circ \mu$. Then $\rho$ is equivalent to $(\tau+\tau')\circ\mu$.
Applying \eqref{eq1} and \eqref{eq2} it follows
that for $\sigma\in\Gal(\bar{F}/F)$ and $g\in G(\bar{F})$ one has 
\begin{align*}
\rho(\sigma(g))=x\tilde{\tau}(a_\sigma)(\tau+\tau')(\sigma(\mu(g)))\tilde{\tau}
(a_\sigma^{-1})x^{-1}=\sigma(x)\sigma\left((\tau+\tau')(\mu(g))\right)\sigma(x)^
{-1}=\sigma(\rho(g)),
\end{align*}
where we used that $\tau+\tau'$ is defined over $F$ and hence commutes with
$\Gal(\bar F/F)$. 
Thus $\rho$ commutes with $\Gal(\bar{F}/F)$ and thus it is defined over $F$. 
\end{proof}

Now we may fix an embedding of $\SO^0(p,q)$ into $\SO_{B}(\C)$ such that 
the representations of $\SO^0(p,q)$ with highest weight
$m_1e_1+\dots+m_{n+1}e_{n+1}$ 
are the restrictions to $\SO^0(p,q)$ of the representation of $\SO_{B}(\C)$
with highest weight $m_1f_1+\dots+m_{n+1}f_{n+1}$.

The following proposition is certainly well known and was used already 
by Bergeron and Venkatesh \cite[section 8.1]{BV}. However, for 
the sake of completeness we include a proof here. If $V$ is a finite-dimensional
$F$-vector
space, let $V^\sigma$ be the $F$-vector space with scalar-multiplication $a\cdot
v:=\sigma(a) v$, $a\in F$, $v\in V$.

\begin{lem}\label{LemQ}
Let $G'$ be an algebraic group defined over $\Q$. Let $V$ be a
finite-dimensional $F$-vector space and let $\rho:G'\to \Gl(V)$ be 
a representation defined over $F$. Then
$\tilde{\rho}:=\prod_{\sigma\in\Sigma}\rho^{\sigma^{-1}}$ is defined over $\Q$,
where 
$\rho^{\sigma^{-1}}$ is regarded as an $F$-rational representation of $G'$ on
$V^{\sigma}$.
\end{lem}
\begin{proof}
Each $\sigma\in\Sigma$ acts on $\prod_{\sigma\in\Sigma}V^\sigma$ as a 
$\sigma$-linear automorphism by permuting the factors. The corresponding 
$\Q$-structure of $\prod_{\sigma\in\Sigma}V^\sigma$ is $V$, regarded as a
$\Q$-vector
space and embedded diagonally into $\prod_{\sigma\in\Sigma}V^\sigma$ . Now it is
easy to see that $\tilde{\rho}$ commutes with the action of $\Sigma$ on 
$G'$ and the action of $\Sigma$ on $\Gl(\prod_{\sigma\in\Sigma}V^\sigma)$
associated 
to this $\Q$-structure. Thus $\tilde{\rho}$ is defined over $\Q$. 
\end{proof}

Let $G^\prime(\R)^0$ be the connected component of $1\in G^\prime(\R)$ and
$G^\prime(\R)^c:= \prod_{\sigma\in\Sigma\setminus\{1\}}\SO(p+q)$. Then we have
\[
G^\prime(\R)^0\cong\SO^0(p,q)\times G^\prime(\R)^c.
\]
Let $\theta$ be the standard Cartan-involution of $\SO^0(p,q)$. Then
$\theta\otimes \Id_{G^\prime(\R)^c}$ is a Cartan involution of $G'(\R)^0$ which
we 
continue to denote by $\theta$. By $\Rep(G'(\R)^0)$ we denote 
the finite-dimensional irreducible complex representations of $G'(\R)^0$. For
$\tau\in\Rep(G'(\R)^0)$, let
$\tau_\theta$ be the element of $\Rep(G'(\R)^0)$ defined by 
$\tau_\theta:=\tau\circ\theta$. 

\begin{prop}\label{Proprep}
There exists a sequence $\rho(m)$ of $\Q$-rational representations of $G'$ on
finite-dimensional
$\Q$-vector spaces $V_{\rho(m)}$ such that 
\begin{enumerate}
\item For the decomposition
\begin{equation}\label{decrho}
\rho(m)=\bigoplus_{\tau\in\Rep(G'(\R)^0)}[\rho(m):\tau]\tau,
\end{equation}
$[\rho(m):\tau]\in\mathbb{N}^0$ of $\rho(m)$, regarded as a complex
representation 
of $G'(\R)^0$ on the vector space $V_{\rho(m)}\otimes_\Q\C$,  
into irreducible representations of $G'(\R)^0$ one has $\tau\neq\tau_\theta$ for
each
$\tau\in\Rep(G'(\R)^0)$ with 
$[\rho(m):\tau]\neq 0$. 
\item The dimension $\dim(V_{\rho(m)})$ is a polynomial in $m$ with leading
term 
\begin{align*}
\dim(V_{\rho(m)})=C\,d\,m^{d n(n+1)/2} + O(m^{d(n+1)/2-1}),
\end{align*}
 as $m\to\infty$, where $C>0$ is a constanst which depends only on $n$. 
\end{enumerate}
\end{prop}

\begin{proof}
The Galois group $\Sigma$ acts on $\prod_{\sigma\in\Sigma}G^\sigma$ as 
follows. For $g\in\prod_{\sigma\in\Sigma}G^\sigma$ 
and $\sigma\in\Sigma$ we denote the projection of $g$ to $G^\sigma$ by
$g_\sigma$.
Then for $\sigma,\sigma'\in\Sigma$ one has 
\begin{align*}
(\sigma g)_{\sigma'}=\sigma(g_{\sigma^{-1}\sigma'}).
\end{align*}
Now assume that for each $\sigma\in\Sigma$ we are given a finite-dimensional
$F$-vector space $V_{\rho(\sigma)}$ and a representation $\rho(\sigma)$ of
$G^\sigma$ on $V_{\rho(\sigma)}$, defined 
over $F$. Then the tensor-product
\begin{align*}
\rho:=\bigotimes_{\sigma\in\Sigma}\rho(\sigma)
\end{align*}
is a representation of $\prod_{\sigma\in\Sigma}G^\sigma$ on
$\bigotimes_{\sigma\in\Sigma}V_{\rho(\sigma)}$ and it follows that for
$\sigma'\in\Sigma$ one has
\begin{align}\label{rhosigma}
\rho^{\sigma'}=\bigotimes_{\sigma\in\Sigma}\rho(\sigma'^{-1}\sigma)^{\sigma'}.
\end{align}
Now if $n$ is even, for $m\in\mathbb{N}$ we let $\tau(m)$ be the representation
of $G$ over $\bar{F}$ 
of highest weight $2me_1+\dots+2me_{n+1}$. If $n$ is odd, we let $\tau(m)$ be
the representation 
of highest weight $me_1+\dots+me_{n+1}$. Then $\tau(m)$ and $\tau(m)_\theta$
descend to representations 
of $PG$. Thus by Lemma \ref{Lemma2} there exists a representation of $G$ over 
$F$ which over $\bar{F}$ is equivalent to $\tau(m)+\tau(m)_\theta$. Thus if we 
define $\rho_0(m)$ by 
\begin{align}\label{rho0}
\rho_0(m):=\bigotimes_{\sigma\in\Sigma}(\tau(m)+\tau(m)_\theta)^\sigma,
\end{align}
then $\rho_0(m)$ is defined over $F$ and by \eqref{rhosigma}, $\rho_0(m)$ is
equivalent 
to $\rho_0(m)^\sigma$ for each $\sigma\in\Sigma$. Hence if we let $\rho(m)$ 
be the direct sum of $d$ copies of $\rho_0(m)$ then by Lemma \ref{LemQ}
$\rho(m)$
is defined over $\Q$. Each irreducible component of $\rho(m)|_{G^\prime(\R)^0}$,
regarded as a complex representation 
of $G'(\R)^0$ on $V_{\rho(m)}\otimes_\Q\C$, is 
of the form
$\tau(m)\otimes\pi$, or $\tau(m)_\theta\otimes\pi^\prime$, where $\pi$ and
$\pi^\prime$ are irreducible representations of $G^\prime(\R)^c$. 
Since $\tau(m)$ and $\tau(m)_\theta$ are 
not $\theta$-invariant, the same holds for each irreducible component of
$\rho(m)|_{G^\prime(\R)^0}$. 
This proves the first statement. The second statement follows from Weyl's
dimension 
formula. 
\end{proof}

We can now turn to the proof of Theorem \ref{Thrm1}. Let $\Delta$ be the
diagonal embedding of $G$ into $\prod_{\sigma\in\Sigma}G^\sigma$. Then  
we can choose the isomorphism $\alpha$ in \eqref{eqalpha} such that
$\alpha(G'_\Z)=\Delta(G_{\mathcal{O}_F})$. Let $\Gamma\subset G_{\mathcal{O}_F}$
be a subgroup of finite index. Via the isomorphism $G(\R)\cong \SO(p,q)$ we
identify $\Gamma$ with a subgroup of $\SO(p,q)$. We choose $\Gamma$ such that
it is torsion free and is contained in $\SO^0(p,q)$. By Lemma \ref{Lemma1}, 
$\Gamma$ is a cocompact lattice in $\SO^0(p,q)$. Let $\Gamma^\prime=
\Delta(\Gamma)$. Since $\Gamma$ and $\Gamma^\prime$ are isomorphic, it suffices
to prove the statements of Theorem \ref{Thrm1} for $\Gamma^\prime$. 

The group $K_0:=\SO(p)\times\SO(q)$ is a a maximal compact subgroup of
$\SO^0(p,q)$. Put
\begin{align*}
K':=K_0\times\prod_{\substack{\sigma\in\Sigma\\ \sigma\neq
1}}\SO(p+q).
\end{align*}
Then $K^\prime$ is a maximal compact subgroup of $G^\prime(\R)^0$. 
Put $\tilde{X}':=G'(\R)^0/K'$ and $X':=\Gamma'\backslash \tilde{X}'$. 
Let $(\rho(m),V_{\rom})$ be the sequence of $\Q$-rational representations of 
$G'$ of Proposition \ref{Proprep}. Since each $\rho(m)$ is defined 
over $\Q$, there exists a free $\Z$-module $M_{\rom}$ in $V_{\rom}$ 
which is stable under $\rho(m)(\Gamma')$ and such that 
$M_{\rom}\otimes_\Z\Q\cong V_{\rom}$, see for example \cite[page 173]{PR}. 
Let  $V_{\rom}^\C:=V_{\rom}\otimes_\Q\C$. 
Then the restriction of $\rho(m)$ to $\Gamma'$ induces the flat 
complex vector bundle 
\begin{align*}
E_{\rom}:=\tilde{X'}\times_{\rho(m)|_{\Gamma'}}V_{\rom}^\C
\end{align*}
over $X'$. 
The decomoposition \eqref{decrho} of $\rho(m)$ induces a 
corresponding decomposition of $E_{\rom}$ into the direct sum of complex vector
bundles 
associated to the
restriction to $\Gamma^\prime$ of irreducible finite-dimensional
representations 
$\tau$ of $G'(\R)^0$. By Proposition \ref{Proprep}, each $\tau$ with
$[\rho(m):\tau]\neq 0$ satisfies 
$\tau\neq\tau_{\theta}$ and thus by \cite[Chapter VII, Theorem 6.7]{BW} one has
\begin{align}\label{vancoho}
H^*(X';E_{\rom})=0,
\end{align}
where $H^*(X';E_{\rom})$ denotes the de Rham cohomology with coefficients in
$E_{\rom}$. 
Chose a Hermitian fibre metric in $E_{\rom}$. Let $T_{X'}(\rho(m))$ 
be the analytic torsion of $X^\prime$ and $\rho(m)$ (see \eqref{anator}). 
It follows from 
\eqref{vancoho} that $T_{X'}(\rho(m))$ is metric independent 
\cite[Corollary 2.7]{Mu1}. Moreover $H^*(\Gamma^\prime,M_{\rom})$ is a finite
abelian group and by \eqref{cohotor} we have
\begin{equation}\label{cohotor1}
\log T_{X^\prime}(\rho(m))=\sum_{q=0}^n
(-1)^{q+1}\log|H^q(\Gamma^\prime,M_{\rom})|.
\end{equation}
This equality reduces the proof of Theorem \ref{Thrm1} to the study of the 
asymptotic behavior of $T_{X'}(\rho(m))$ as $m\to\infty$, which is exactly the
problem that has been dealt with in \cite{MP}. Since $\{\rho(m)\}$ is not a 
sequence of representations that has been considered in \cite{MP}, we cannot 
apply the results of \cite{MP} directly. We first need to reduce it to a case
to which \cite{MP} can be applied.

Let $\rho_0(m)$ be defined by \eqref{rho0} and let $T_{X^\prime}(\rho_0(m))$ be
the
corresponding analytic torsion. Since $\rho(m)$ is the direct sum 
of $d$ copies of $\rho_0(m)$, we get
\begin{align*}
\log T_{X'}(\rho(m))=d\log T_{X'}(\rho_0(m))
\end{align*}
Now let 
$T_{X'}^{(2)}(\rho_0(m))$ be the $L^2$-torsion with respect to $\rho_0(m)$ (see 
\cite[section 5]{MP}). If we apply \cite[Proposition 1.2]{MP} to the irreducible
components of $\rho_0(m)$, it follows that there exists $c>0$ such that
\begin{align*}
\log{T_{X'}(\rho_0(m))}=\log{T_{X'}^{(2)}(\rho_0(m))}+O(e^{-cm}),
\end{align*}
as $m\to\infty$. Using the definition of $\rho_0(m)$ by \eqref{rho0},  
\cite[(5.21)]{MP} and \cite[Proposition 5.3]{MP}, it follows  that
\begin{align*}
\log T_{X'}^{(2)}(\rho_0(m))=
\left(\log T_{X^\prime}^{(2)}(\tau(m))+\log
T_{X^\prime}^{(2)}{\tau(m)_\theta}\right)(2\dim\tau(m))^{d-1}.
\end{align*}
If $C_{p,q}$ is as in \eqref{Cpq}, then by \cite[Proposition 6.7]{MP} one has
\begin{align*}
\log{T_{X^\prime}^{(2)}}(\tau(m))=\log{T_{X^\prime}^{(2)}}(\tau(m)_\theta)=C_{p,
q}
\vol(X^\prime)m\dim\tau(m)+O(\dim(\tau(m)),
\end{align*}
as $m\to\infty$. Thus putting everything together we obtain 
\begin{align*}
\log{T_{X'}(\rho(m))}=C_{p,q}\vol(X^\prime)m\dim(\rho(m))+
O(\dim(\rho(m))),
\end{align*}
as $m\to\infty$.  Since $X\cong X^\prime$ and $H^*(\Gamma,M_{\rho(m)})\cong
H^*(\Gamma^\prime,M_{\rho(m)})$, Theorem \ref{Thrm1} follows from 
\eqref{cohotor1} and the second statement of Proposition \ref{Proprep}.

\section{Arithmetic subgroups of $\SL_3(\R)$}\label{secsldrei}
Let $D$ be a nine-dimensional division algebra over $\Q$. Then by the
Brauer-Hasse-Noether theorem \cite{Ro}, 
$D$ is a cyclic algebra for a cubic extension $L$ of $\Q$. Moreover,
$L$ 
splits $D$, i.e. there exists an isomorphism of $L$-algebras
\begin{align}\label{phi}
\phi:D\otimes_\Q L\cong \Mat_{3\times 3}(L).
\end{align}
Thus for $x\in D$ the reduced norm $N(x)$ is given by $N(x):=\det(\phi(x\otimes
1))$. 
Now let $G:=\SL_1(D)$, where
\begin{align*}
\SL_1(D):=\{x\in D\colon  N(x)=1\}.
\end{align*}
Then by \cite[2.3.1]{PR}, $G$ has a canonical structure of an algebraic group
defined
over $\Q$. We regard $\SL_3$ as an algebraic group over $\Q$. 
The isomorphism $\phi$ from \eqref{phi} induces an isomorphism
\begin{align*}
\phi:G(L)\cong \SL_3(L),
\end{align*}
i.e. $G$ is a form of $\SL_3$ over $L$. Moreover, the following Lemma holds. 
\begin{lem}\label{Lem1sl3}
Let $\rho$ be a $\Q$-rational representation of $\SL_3$. Then there exists a 
$\Q$-rational representation of $G$ which over $L$ is equivalent to
$\rho\circ\phi$.
\end{lem}
\begin{proof}
By \cite[Proposition 2.17]{PR}, $G$ is an inner form of $\SL_3$. Thus the proof
of Lemma 3.1 in
\cite{MaM} 
can be generalized without difficulties to prove the Lemma. 
\end{proof}
Let $\mathfrak{o}$ be an order in $D$, i.e., $\mathfrak{o}$ is a free
$\Z$-submodule of $D$ which 
is generated by a $\Z$-base of $D$ and which is also a subring of $D$. Put 
\begin{align*}
\mathfrak{o}^1:=\{x\in\mathfrak{o}\colon N(x)=1\}.
\end{align*}
The left regular representation of $D$ on itself induces a $\Q$-rational
representation 
of $G$ on $D$, see \cite[2.3.1]{PR} and the stabilizer of $\mathfrak{o}$ is
$\mathfrak{o}^1$. 
Thus $\mathfrak{o}^1$ is arithmetic subgroup of
$G(\Q)$. Put 
\begin{align*}
\Gamma:=\phi(\mathfrak{o}^1).
\end{align*}
Then $\Gamma$ is an arithmetic subgroup of $\SL_3(\R)$. Moreover, the following
lemma holds.
\begin{lem}
The group $\Gamma$ is a cocompact subgroup of $\SL_3(\R)$. 
\end{lem}
\begin{proof}
By \cite[Proposition 2.12]{PR}, $G$ is anisotropic over $\Q$. Thus the
proposition follows from 
\cite[Lemma 11.4, Theorem 11.8]{BH}.
\end{proof}

Let $T$ be the standard maximal torus in $\SL_3$ consisting of the diagonal
matrices of determinant $1$. Then $T$ is defined over $\Q$ and is $\Q$-split. 
Let $\tL$ be the Lie-algebra of $T(\R)$ consisting of all diagonal matrices 
of trace $0$. Let $e_i\in\tL^*$ be defined by
$e_i(\diag(t_1,t_2,t_3)):=\sum_{j=1}^3\delta_{i,j}t_j$. 
Then with respect to the standard odering of the roots of $\tL_\C$ in
$\mathfrak{sl}_{3,\C}$ the fundamental weights
$\omega_1,\omega_2\in\tilde{\tL}_\C^*$  
are given by
\begin{align}\label{fweights}
\omega_1=\frac{2}{3}(e_1-e_2)+\frac{1}{3}(e_2-e_3)
;\quad \omega_2=\frac{1}{3}(e_1-e_2)+\frac{2}{3}(e_2-e_3).
\end{align}
The finite-dimensional irreducible representations $\tau$ of $\SL_3(\R)$ are
parametrized by 
their highest weights $\Lambda_\tau=m_1\omega_1+m_2\omega_2$. If $\theta$ is 
the standard Cartan involution of $\SL_3(\R)$, then the highest weight of the
representation 
$\tau_\theta:=\tau\circ\theta$ is given by
$\Lambda_{\tau_\theta}=m_2\omega_1+m_1\omega_2$.

\begin{prop}\label{Propsldrei}
Let $\tau$ be a finite-dimensional irreducible representation of $\SL_3(\R)$ on 
a finite-dimensional vector space $V_\tau$. Then there exists a lattice $M$ 
in $V_\tau$ which is invariant under $\tau(\Gamma)$.  
\end{prop}
\begin{proof}
Since $T$ is $\Q$-split $\tau$ is defined over $\Q$  \cite{Tits}[Proposition
2.3]. 
By Lemma \ref{Lem1sl3}, there exists a rational representation $\tau'$ of $G$
on 
a finite-dimensional $\Q$-vector space $V(\tau')$  
which over $L$ is equivalent to $\tau\circ\phi$. Since $\mathfrak{o}^1$ is 
an arithmetic subgroup of $G(\Q)$, there exists a lattice in $V(\tau')$ which 
is stable under $\tau'(\mathfrak{o}^1)$, see for example \cite[page 173]{PR}.
Since $\Gamma=\phi(\mathfrak{o}^1)$, 
the Proposition follows. 
\end{proof}

We can now turn to the proof of Theorem \ref{Thrm2}. Let $\widetilde X=
\SL_3(\R)/\SO(3)$ and $X=\Gamma\bs\widetilde X$, where $\Gamma\subset \SL_3(\R)$
is an arithmetic subgroup as above. Chose a $\SL_3(\R)$-invariant metric on 
$\widetilde X$ and equip $X$ with the induced metric. Let $\Lambda\in\hL_\C^*$
be a highest
weight. Assume that  $\Lambda$ satisfies $\Lambda\neq\Lambda_\theta$. Then the 
same holds for each weight $m\Lambda$, $m\in\N$. Let $\tau_\Lambda(m)$ be the 
irreducible finite-dimensional 
representation on $V_\Lambda(m)$ with highest weight $m\Lambda$. Let 
$E_{\tau_\Lambda(m)}$ be the flat vector bundle over $X$ associated to 
$\tau_\Lambda(m)$. By \cite[Chapter VII, Theorem 6.7]{BW} we have 
\begin{equation}\label{vancoh3}
H^*(X,E_{\tau_\Lambda(m)})=0.
\end{equation}
Let $T_X(\tau_\lambda(m))$ be the analytic torsion with respect to any
Hermitian fibre metric in $E_{\tau_\Lambda(m)}$. By \eqref{vancoh3} and
\cite[Corollary 2.7]{Mu1}, $T_X(\tau_\Lambda(m))$ is independent of the 
choice of metrics on $X$ and in $E_{\tau_\Lambda(m)}$. 
Let $M_\Lambda(m)\subset V_\Lambda(m)$ be an arithmetic $\Gamma$-module, which
exists by Proposition \ref{Propsldrei}. By \eqref{vancoh3},
$H^*(X,M_\Lambda(m))$
is a finite abelian group and by \eqref{cohotor} we have
\begin{equation}\label{cohotor3}
\log T_X(\tau_\lambda(m))=\sum_{q=0}^5(-1)^{q+1}\log |H^q(X,M_\Lambda(m))|.
\end{equation}
Using Theorem 1.1 and Corollary 1.5 of \cite{MP}, the proof of  
Theorem \ref{Thrm2} follows.

\end{document}